\documentclass[12pt]{article}
\usepackage{amsmath}
\usepackage{amsfonts}
\usepackage{amssymb}
\usepackage{amsthm}
\usepackage{stackrel}

\usepackage{color}

\usepackage{figsize}

\usepackage{graphics}

\newtheorem{thm}{Theorem}[section]

\newcommand{\R}{{\rm I}\kern-0.18em{\rm R}}
\newcommand{\1}{{\rm 1}\kern-0.25em{\rm I}}
\newcommand{\E}{{\rm I}\kern-0.18em{\rm E}}
\newcommand{\p}{{\rm I}\kern-0.18em{\rm P}}

\makeatletter

\def\@fnsymbol#1{\ensuremath{\ifcase#1\or a\or b\or c\or d\or \e\or f\or *\dagger 	\or \ddagger\ddagger \else\@ctrerr\fi}}
\title{Independence of Linear Statistics with Random Coefficients and Characterizations of Geometric and Poisson Distributions}
\author{Lev B. Klebanov\footnote{Department of Probability and Mathematical Statistics, Charles University, Prague, Czech Republic. e-mail: lev.klebanov@mff.cuni.cz}}
\date{}
\begin{document}
\maketitle

\begin{abstract}
There is given a characterization of the geometric distribution by the independence of linear forms with random coefficients. The result is a discrete analog of the corresponding theorem on exponential distribution. The property of linear statistics independence is also a characterization of Poisson law.

\vspace{0.3cm}
\end{abstract}

\section{Introduction}\label{sec1}
\setcounter{equation}{0} 
Many different characterizations of the distribution are known (see \cite{KLR}). An essential part of them is connected to the characterizations property of independence of statistics, especially by that of linear forms. The main result here is the linear forms are independent for Gaussian distribution only. The property allows explaining the appearance of Maxwell distribution in Physics. Recently, there were published facts on the independence of linear forms with random coefficients (see \cite{K1, KV}). They lead to the characterization of exponential and hyperbolic secant distributions. Unfortunately, it was not clear if it is possible to use the independence property to characterize some discrete distributions.  Here we modify previous methods to obtain characterizations of geometric and Poisson distribution. We hope the methods will lead to characterizations of other discrete distributions. As far as we know, there are only a few characterizations of the distributions of positive integer-valued random variables. Among them let us note the characteristic properties of geometric distribution by the independence of linear forms of order statistics \cite{GK}. However, the structure of the forms used here is absolutely different.

\section{Main results}\label{sec1} 
\setcounter{equation}{0} 

In the paper, \cite{KV} there was given a characterization of exponential distribution by the independence of linear forms
\[ S_1= (1-p) aX+\varepsilon(p) a Y \quad \text{and} \quad S_2=pbX+(1-\varepsilon(p))bY. \]
Here $X$ and $Y$ are independent identically distributed (i.i.d) positive random variables, $p\in (0,1)$, and $\varepsilon(p)$ is independent of the pair $(X,Y)$ Bernoulli random variable with the parameter $p$. Below we give a similar result for a geometric random variable. However, the multiplication of $X$ by $p$ or by $1-p$ leads out the class of integer-valued random variables. Therefore, we have to use a thinning operator instead of multiplication. 

Let us give a precise formulation. Suppose $X$ and $Y$ are i.i.d. random variables taking non-negative integer values. Suppose that $p\in (0,1)$ is fixed, and $\varepsilon_p$ is independent of the pair $(X,Y)$ Bernoulli random variable with the parameter $p$. Let $\{\varepsilon_j(1-p),\; j=1,2,\ldots\}$ be a sequence of i.i.d. Bernoulli random variables with parameter $1-p$ and independent on $X,Y,\varepsilon(p)$. Define the forms 
\begin{equation}\label{eq1}
L_1=\widetilde{X}_{1-p}+\varepsilon(p)Y \quad \text{and}\quad L_2=\widetilde{X}_{p}+\bigl(1-\varepsilon(p)\bigr)Y.
\end{equation}
Here $\widetilde{X}_{1-p}=\sum_{j=1}^{X}\varepsilon_j(1-p)$ and $\widetilde{X}_{p}=\sum_{j=1}^{X}\bigl(1-\varepsilon_j(1-p)\bigr)$.

\begin{thm}\label{th1}
The forms $L_1$ and $L_2$ defined above are stochastical independent if and only if $X$ has a geometric distribution.
\end{thm}
\begin{proof}
Consider joint probability generating function of $L_1$ and $L_2$.
\begin{eqnarray}\label{eq2}
\nonumber
\E \bigl(u^{L_1}v^{L_2}\bigr)=\E \bigl( u^{\widetilde{X}_{1-p}+\varepsilon(p)Y} v^{\widetilde{X}_{p}+(1-\varepsilon(p))Y}\bigr)=\\
\mathcal{P}((1-p)u+pv)\Bigl(p \mathcal{P}(u) + (1-p)\mathcal{P}(v) \Bigr),
\end{eqnarray}
Forms $L_1$, and $L_2$ are independent if and only if the right-hand side of (\ref{eq2}) may be written as a product of two probability generating functions depending on $u$ and $v$ separately.

Let us verify the forms are independent in the case of exponentially distributed $X$. Really, in this case, we have
\[  \mathcal{P}(z) =\frac{a-1}{a-z} \]
for any $a>1$. Substitution of this function on the right-hand side of (\ref{eq2}) gives us the product
\[ \frac{(a-1)^2}{(a-u)(a-v)}.  \]
It means the forms are independent for exponentially distributed $X$.

Let us prove the inverse statement. Independence of $L_1$ and $L_2$ holds if and only if their joint probability generating function is a product of corresponding marginal generating functions:
\begin{eqnarray}\label{eq3}
\mathcal{P}\bigl((1-p)u+pv\bigr)\Bigl(p \mathcal{P}(u) + (1-p)\mathcal{P}(v) \Bigr)=\nonumber\\ \mathcal{P}\bigl((1-p)u+p\bigr)\Bigl(p \mathcal{P}(u) + (1-p) \Bigr)\mathcal{P}\bigl((1-p)+pv\bigr)\Bigl(p + (1-p)\mathcal{P}(v) \Bigr). 
\end{eqnarray}
The relation (\ref{eq3}) holds for all $u$ and $v$ such that $|u|\leq 1$, $|v|\leq 1$. 

Setting here $u=v=0$ we obtain
\[  \mathcal{P}^2(0) = \mathcal{P}(p) \Bigl( \mathcal{P}(0)p+1-p\Bigr) \mathcal{P}(1-p)\Bigl(p+(1-p)\mathcal{P}(0) \Bigr).\]
However, $p\in (0,1)$ and from  the definition of probability generating function we see $\mathcal{P}(p) >0$ and $\mathcal{P}(1-p) >0$. Therefore $\mathcal{P}(0) >0$. Because any probability generating function is analytic in the disc $|z|\leq 1$ on the complex plain $\mathcal{P}(z)$ is uniquely defined by the sequence of its derivatives at point $z=0$.

Rewrite (\ref{eq3}) in the form
\begin{equation}\label{eq4}
\mathcal{P}\bigl((1-p)u+pv\bigr)\Bigl(p \mathcal{P}(u) + (1-p)\mathcal{P}(v) \Bigr)=H_1(u)H_2(v).
\end{equation}
Taking the logarithm and differentiating with respect to $v$ both sides of (\ref{eq4}), we obtain at the point $v=0$ that
\begin{equation}\label{eq5}
p\frac{\mathcal{P}^{\prime}\bigl((1-p)u\bigr)}{\mathcal{P}\bigl((1-p)u\bigr)}+\frac{(1-p)\mathcal{P}^{\prime}(0)}{p\mathcal{P}(u)+(1-p)\mathcal{P}(0)}=C,
\end{equation}
where $C=H_2^{\prime}(0)/H_2(0)$. From (\ref{eq5}) follow two facts: 
\begin{enumerate}
\item The value $\mathcal{P}^{\prime}(0)$ may be arbitrary;
\item For any $k>1$ the value of $\mathcal{P}^{k}(0)$ is uniquely determined by previous values  $\mathcal{P}^{k-1}(0), \mathcal{P}^{k-2}(0), \ldots , \mathcal{P}(0)$ that is by two parameters $\mathcal{P}(0)$ and $\mathcal{P}^{\prime}(0)$.
\end{enumerate} 
However, one of these parameters is fixed in view of the condition 
\[  \sum_{k=0}^{\infty}\frac{\mathcal{P}^{k}(0)}{k!} = 1\]
and the solution of (\ref{eq2}) may depend on one parameter only. Such solution is $\mathcal{P}(z) =(a-1)/(a-z)$.
\end{proof}

Let us now proceed to the characterization of Poisson distribution. Let $X,Y$ are i.i.d. non-negative integer-valued random variables. Suppose that
$\{\varepsilon_j(p),\; j=1,2,\ldots\}$ and $\{\widetilde\varepsilon_j(q),\; j=1,2,\ldots\}$ are independent with each other and with the pair $(X,Y)$ sequences of Bernoulli random variables with parameters $p,q \in (0,1)$. Consider linear forms $K_1$ and $K_2$ of $X$ and $Y$:
\begin{eqnarray}\label{eq6}
\widetilde{X}_{1-p}=\sum_{j=1}^{X}\varepsilon_j(1-p), \quad \widetilde{X}_{q}=\sum_{j=1}^{X}\bigl(1-\widetilde\varepsilon_j(1-q)\bigr) \nonumber\\
\widetilde{Y}_{1-q}=\sum_{j=1}^{Y}\widetilde\varepsilon_j(1-q), \quad \widetilde{Y}_{p}=\sum_{j=1}^{Y}\bigl(1-\varepsilon_j(1-p)\bigr) \nonumber\\
K_1=\widetilde{X}_{1-p}+\widetilde{Y}_{1-q}, \quad K_2=\widetilde{X}_{q}+\widetilde{Y}_{p}.
\end{eqnarray}
\begin{thm}\label{th2}
Forms (\ref{eq6}) are independent if and only if $X$ has Poisson distribution.
\end{thm}
\begin{proof}
The joint probability generating function of $K_1$ and $K_2$ is
\begin{eqnarray}\label{eq7}
\E (u^{K_1}v^{K_2}) = \E \Bigl(u^{\widetilde{X}_{1-p}}v^{\widetilde{Y}_{p}} u^{\widetilde{X}_{q}}v^{\widetilde{Y}_{1-q}}\Bigr)
=\nonumber\\ 
\mathcal{P}((1-p)u+pv)\mathcal{P}((1-q)u+qv).
\end{eqnarray}
Probability generating function of Poisson distribution has form 
\[ \mathcal{P}_o(z) = \exp\{\lambda(z-1) \}. \]
Substituting this function on the right-hand side of (\ref{eq7}) leads to a product of functions depending on $u$ and $v$ separately.

The rest of the proof is similar to that of the Theorem \ref{th1}; we omit it.
\end{proof}

\section{Conclusions} 

Characterization of distributions by the property of independent statistics is an interesting part of Probability. Basically, it is purely theoretical interest. However, characterizations of distributions can also be useful in statistical hypotheses testing (see \cite{N}). Obtained here characterizations use rather simple statistics and there is a hope they may also be used for the construction of statistical tests.

\section*{\small{ACKNOWLEDGEMENT}}

The work was partially supported by Grant GA\v{C}R 19-28231X.

\end{document}